\documentclass[leqno,11pt]{amsart}

\usepackage{setspace}
\usepackage[english]{babel}
\usepackage{mathtools}
\usepackage{mathrsfs}
\usepackage{enumitem}
\usepackage{amsthm}
\usepackage{amssymb}
\usepackage{bbm}

\newtheorem{theorem}{Theorem}[section]

\newtheorem{lemma}[theorem]{Lemma}
\newtheorem{proposition}[theorem]{Proposition}

\newtheorem{fact}[theorem]{Fact}
\newtheorem{claim}[theorem]{Claim}
\newtheorem*{thm*}{Theorem}

\theoremstyle{definition}
\newtheorem{definition}[theorem]{Definition}

\theoremstyle{remark}
\newtheorem{remark}{Remark}

\newtheorem{question}{Question}

\def\forces{\Vdash}

\def\ZFC{\mathsf{ZFC}}

\def\baire{\omega^\omega}

\def\mfb{\mathfrak b}
\def \mfd{\mathfrak{d}}
\def \mfi{\mathfrak{i}}

\def\CH {\mathsf{CH}}
\def\Q{\mathbb Q}
\def\P{\mathbb P}

\def\cc{2^{\aleph_0}}
\def\mfp{\mathfrak{p}}

\def\mfu{\mathfrak{u}}

\usepackage[margin=1.3in]{geometry}

\begin{document}

\title{Generic Selective Independent Families}

\author[Fischer]{Vera Fischer}
\address[V. ~Fischer]{Institute of Mathematics, University of Vienna, Kolingasse 14-16, 1090 Vienna, Austria}
\email{vera.fischer@univie.ac.at}

\author[Switzer]{Corey Bacal Switzer}
\address[C.~B.~Switzer]{Institute of Mathematics, University of Vienna, Kolingasse 14-16, 1090 Vienna, Austria}
\email{corey.bacal.switzer@univie.ac.at}

\thanks{\emph{Acknowledgments:} The authors would like to thank the
Austrian Science Fund (FWF) for the generous support through grant number Y1012-N35.}
 %%most common subj class for me

\date{}

\maketitle

\begin{abstract}
    We prove that the generic maximal independent family obtained by iteratively forcing with the Mathias forcing relative to diagonalization filters is densely maximal. Moreover, by choosing the filters with some care one can ensure the family is selective and hence forcing indestructible in a strong sense. Using this we prove that under $\mfp = \cc$ there are selective independent families and also we show how to add selective independent families of any desired size.
\end{abstract}

\section{Introduction}
Recall that a family $\mathcal I \subseteq [\omega]^\omega$ is \emph{independent} if for all finite, disjoint $\mathcal A, \mathcal B \subseteq \mathcal I$ the set $\bigcap \mathcal A \setminus \bigcup \mathcal B$ is infinite. Such a family is a \emph{maximal} independent family or a m.i.f. if it is maximal with this property. Denote by $\mathsf{FF}(\mathcal I)$ the collection of finite functions $h:\mathcal I \to 2$ and for all $h \in \mathsf{FF}(\mathcal I)$ let $\mathcal I^h := \bigcap_{A \in {\rm dom}(h)} A^h$ where $A^h = A$ if $h(A)=0$ and $A^h = \omega \setminus A$ if $h(A) = 1$. Thus $\mathcal I$ is independent if for all $h \in \mathsf{FF}(\mathcal I)$ $\mathcal I^h$ is infinite. The sets of the form $\mathcal I^h$ are called {\em Boolean combinations}. 

Maximal independent families are one of several important examples of maximal combinatorial sets of reals studied in set theory. %and its environs. 
Other examples include MAD families, MED families and ultrafilter bases. In each case there is an associated cardinal characteristic: the least size of a maximal family of that type. In the case of m.i.f.'s this cardinal is denoted $\mfi$. See \cite{BlassHB} for more information on $\mfi$ and related cardinals. When trying to prove that such a cardinal can be consistently less than the continuum one often needs to construct witnesses which satisfy a stronger maximality condition which can be preserved by iterations of appropriate forcing notions. 

In the case of independent families the associated ``strongly maximal" families are called {\em selective independent families} (defined below). These were first investigated by Shelah in his proof of the consistency of $\mfi < \mfu$ in \cite{Sh92}, and further studies can be found in e.g. \cite{DefMIF, free_seq, FMIdeals, CS}. See in particular \cite{FS23} where the authors proved that such families can be preserved by any countable support iteration of Cohen preserving, proper forcing notions for which each iterand preserves the dense maximality of the family\footnote{All undefined terms which we will use will be defined in Section 2 of this article.}. However several aspects of the combinatorics of such families remain unknown including exactly when such families exist. In particular, it is open whether it is consistent with $\ZFC$ that there are no selective independent families. In this paper we begin the investigation into such questions. Our first main theorem is the following.

\begin{theorem}[see Theorem \ref{p=ccor} below]
    $\mfp = \cc$ implies there are selective maximal independent families. \label{p=c}
\end{theorem}

This result actually follows as a corollary of the main result of the paper. 
%To explain this result we recall briefly Brendle's original proof of the consistency of $\mfi < \cc$.
To each independent family $\mathcal I$, maximal or not, there is an associated {\em diagonalization filter}, $\mathcal F_\mathcal I$, the choice of which is not unique, so that forcing with the associated Mathias forcing $\mathbb M (\mathcal F_\mathcal I)$ adds a real $m$ so that $\mathcal I \cup \{m\}$ is independent but $\mathcal I \cup \{m , y\}$ is not independent for any ground model $y \in [\omega]^\omega$. It follows that a finite support iteration of such Mathias forcing notions of any length of uncountable cofinality, adding the generic real to the independent family at each step, generically adds a maximal independent family. Now we can state the main result of this paper.  

\begin{theorem}
In the finite support iteration described above, the diagonalization filters can be chosen so that the m.i.f. produced at the end is selective. More explicitly, any finite support iteration of diagonalization filters produces a densely maximal family whose density filter is a $P$-filter, and the diagonalization filters can be chosen so that the density filter is a $Q$-filter as well.
\label{mainthm}
\end{theorem}

While the the wording above is somewhat technical, the moral is that the obvious way of producing a m.i.f. generically actually produces one which satisfies a stronger maximality condition making its maximality forcing invariant (for appropriate forcing notions). An immediate corollary of Theorem \ref{mainthm} is the following.

\begin{theorem}[see Theorem \ref{i<c} below]
Let $\kappa \leq \lambda$ be cardinals of uncountable cofinality. It is consistent that $\lambda = \cc$ and there is a selective independent family of size $\kappa$. If moreover $\kappa$ is regular then we can arrange $\mfi = \kappa$ as well. 
\label{mainthm1}
\end{theorem}

Prior to the current work, the above was only known for the special case $\kappa = \aleph_1$.

\medskip
The rest of this paper is organized as follows. In the next section we recall some preliminaries we will need. Section 3 provides the proof of Theorem \ref{mainthm}. Section 4 proves some corollaries and additional results including the proof of Theorems \ref{p=c} and \ref{mainthm1}. The paper closes with some relevant questions and lines for further research. Throughout our notation is mostly standard, conforming to that of \cite{JechST}. For all undefined notions involving cardinal characteristics and set theory of the real line we refer the reader to \cite{BlassHB} or the monograph \cite{BarJu95}.

%\;
\medskip
\noindent {\bf Acknowledgements.} The authors thank Juris St\={e}prans for many helpful conversations on the content relating to this paper as well as allowing us to inlcude Proposition \ref{Juris}. They also thank Oswaldo Guzm\'{a}n for pointing out Lemmas 6.6 and 6.7 of \cite{GS94}.

\section{Preliminaries}
Let $\mathcal I$ be an independent family. We say that $\mathcal I$ is \emph{densely maximal} if for every $X \in [\omega]^\omega$ and every $h \in \mathsf{FF}(\mathcal I)$ there is an $h' \supseteq h$ in $\mathsf{FF}(\mathcal I)$ so that either $\mathcal I^{h'} \setminus X$ or $\mathcal I^{h'} \cap X$ is finite. In other words, $\mathcal I$ is densely maximal if for any $X\notin \mathcal I$ the collection of $h \in \mathsf{FF}(\mathcal I)$ witnessing that $X$ cannot be added to $\mathcal I$ while preserving maximality is dense in the partial order $(\mathsf{FF}(\mathcal I), \supseteq)$.
 
The {\em density filter} of an independent family is the collection of all $X \in[\omega]^\omega$ so that for every $h \in \mathsf{FF}(\mathcal I)$ there is an $h' \supseteq h$ in $\mathsf{FF}(\mathcal I)$ so that $\mathcal I^{h'} \setminus X$ is finite. Denote this filter by ${\rm fil}(\mathcal I)$. Densely maximal independent families are characterized by {\em the partition property}, explained below.

\begin{fact}[The Partition Property]
Let $\mathcal I$ be an independent family. The following are equivalent:
\begin{enumerate}
    \item $\mathcal I$ is densely maximal.
    \item $P(\omega) = {\rm fil}(\mathcal I) \cup \langle \omega \setminus \mathcal I^h\; | \; h \in \mathsf{FF}(\mathcal I)\rangle_{dn}$ where $\langle \mathcal X \rangle_{dn}$ denotes the downward closure of $\mathcal X \subseteq [\omega]^\omega$ under $\supseteq^*$.
\end{enumerate}\label{pp}
\end{fact}

Some more facts about the density filter are listed below. The following are easily verified, see \cite[Lemma 5.5]{BFS21}.
\begin{lemma}\label{lemma1}\hfil
\begin{enumerate}
\item
If $\mathcal I'$ is an independent family and $\mathcal{I}\subseteq\mathcal{I}'$ then $\mathrm{fil}(\mathcal{I})\subseteq\mathrm{fil}(\mathcal{I}')$;
\item If $\kappa$ is a regular uncountable cardinal and $\langle\mathcal{I}_\alpha\mid\alpha<\kappa\rangle$ is a continuous increasing chain of independent families then $\mathrm{fil}(\bigcup_{\alpha<\kappa}\mathcal{I}_\alpha)=\bigcup_{\alpha<\kappa}\mathrm{fil}(\mathcal{I}_\alpha)$;
\item If $\mathcal I$ is an independent family then $\mathrm{fil}(\mathcal{I})=\bigcup\{\,\mathrm{fil}(\mathcal{J})\mid \mathcal{J}\in [\mathcal{I}]^{\leq\omega}\}$.
\end{enumerate}
\end{lemma}

We will need another type of filter associated to an independent family as well, which we define now. 

\begin{definition}
Let $\mathcal I$ be an independent family. A {\em diagonalization filter} for $\mathcal I$ is any filter $\mathcal F$ on $\omega$ which is maximal with respect to the property that all $X \in \mathcal F$ have infinite intersection with every $\mathcal I^h$.
\end{definition}

Note that in general diagonalization filters will not be unique. For instance, if $\mathcal I$ is not maximal then there is some real $y \notin \mathcal I$ so that $\mathcal I \cup \{y\}$ is independent and, by the definition of independence it follows that both $y$ and $\omega\setminus y$ have infinite intersection with every Boolean combination of $\mathcal I$. As such there are diagonalization filters containing both $y$ and $\omega\setminus y$ (which obviously therefore cannot be the same). However, it is an easy but key observation that {\em any} $y$ which has infinite intersection with every Boolean combination of a given independent family $\mathcal I$ must have infinite intersection with $\mathcal I^h \cap Z$ for every element $Z \in {\rm fil}(\mathcal I)$ and every $h \in \mathsf{FF}(\mathcal I)$. Hence if $\mathcal F$ is a diagonalization filter then ${\rm fil}(\mathcal I) \subseteq \mathcal F$. In the case of dense maximality we have the converse. 

\begin{lemma}[Essentially Fischer-Montoya, see \cite{FMIdeals}]
Let $\mathcal I$ be independent. The following are equivalent.

\begin{enumerate}
    \item $\mathcal I$ is densely maximal.
    \item The diagonalization filter is unique and equals the density filter.
    \end{enumerate}  \label{mainlemma}
\end{lemma}

\begin{proof}
Item (1) implies item (2) is Corollary 36 of \cite{FMIdeals}. Let us show that (2) implies (1). Thus suppose that ${\rm fil}(\mathcal I)$ is a diagonalization filter. Note that it must be the unique one since any other diagonalization filter extends it but, being a diagonalization filter it is maximal. We will show that $\mathcal I$ is densely maximal. By the partition property, Fact \ref{pp}, it suffices to show that $$P(\omega) = {\rm fil}(\mathcal I) \cup \langle \omega \setminus \mathcal I^h\; | \; h \in \mathsf{FF}(\mathcal I)\rangle_{dn}$$ 

Suppose that $X \notin {\rm fil}(\mathcal I)$. Since ${\rm fil}(\mathcal I)$ is the unique diagonalization filter for $\mathcal I$, it follows that $X$ is not in any diagonalization filter. This means in particular that it does not have infinite intersection with every Boolean combination of $\mathcal I$ (otherwise we could apply Zorn's Lemma to the filter generated by $X$ to get a diagonalization filter which contains $X$). Therefore there is an $h \in \mathsf{FF}(\mathcal I)$ so that $X$ is almost disjoint from $\mathcal I^h$. But then $X$ is almost included in $\omega \setminus \mathcal I^h$ as needed.
\end{proof}

We will need to force with the Mathias forcing of diagonalization filters so we recall this now.

\begin{definition}
    Let $\mathcal I$ be an independent family and let $\mathcal F$ be a diagonalization filter for $\mathcal I$. Denote by $\mathbb M(\mathcal{F})$ the Mathias forcing relativized to $\mathcal F$. A condition for this forcing notion is a pair $(s, A)$ so that 
    \begin{enumerate}
        \item $s \in 2^{<\omega}$ is the characteristic function of a finite set of natural numbers.
        \item $A \in \mathcal F$
        \item ${\rm min}(A) \geq {\rm dom}(s)$.
    \end{enumerate}
If $(s, A)$ and $(t, B)$ are conditions in $\mathbb M(\mathcal F)$ then we let $(s, A) \leq (t, B)$ just in case:
    \begin{enumerate}
        \item $s \supseteq t$
        \item $A \subseteq B$
        \item If $n \in {\rm dom}(s) \setminus {\rm dom}(t)$ and $s(n) = 1$ then $n \in B$. 
    \end{enumerate}
   
\end{definition}

 If $\mathcal F$ is clear from context, or unimportant we often write $\mathbb M(\mathcal I)$ to emphasize the independent family. If $G \subseteq \mathbb M(\mathcal I)$ is generic then the real $m_G := \{n\; | \; \exists (s, A) \in G \, s(n) = 1\}$ is denoted {\em the Mathias real} for $\mathcal I$ (or $\mathcal F$ depending on the context). It is readily checked that $m$ is an infinite set of natural numbers which {\em diagonalizes} $\mathcal F$ i.e. $m \subseteq^* A$ for each $A \in \mathcal F$ where $X \subseteq^* Y$ means that $X \setminus Y$ is finite. Also $V[G] = V[m_G]$. Moreover we have that $\mathcal I \cup \{m_G\}$ is independent but for any ground model $y \in [\omega]^\omega \cap V$ we have $\mathcal I \cup \{m_G, y\}$ is not independent, see \cite{FMIdeals}. It follows that an iteration with finite support of Mathias forcing relativized to the increasing independent families obtained by adjoining the generic Mathias reals (of length an uncountably cofinal ordinal) will result in a model with a m.i.f. of length the iteration. Let us call such a m.i.f. a \emph{generic m.i.f.}. That such a filter exists, and produces a maximal independent family was first observed by Brendle, see \cite{Hal17}. Indeed by performing such an iteration of length $\aleph_1 < \cc$ Brendle obtained the first model of $\mfi < \cc$. 

As stated in the introduction, the main goal of this paper is to improve this result by showing that such an iteration produces a {\em selective} maximal independent family. Selectivity is a further strengthening of dense maximality. We recall some more definitions.

\begin{definition}
Let $\mathcal F$ be a family of subsets of $\omega$. We say that:
\begin{enumerate}
\item
$\mathcal F$ is a $P${\em -set} if every countable family $\{A_n \; | \; n < \omega\} \subseteq \mathcal F$ has a psuedointersection $B \in \mathcal F$, i.e. $B \subseteq^* A_n$ for all $n < \omega$,
\item
$\mathcal F$ is a $Q$-{\em set} if given every partition of $\omega$ into finite sets $\{I_n \; |\; n < \omega\}$ there is a {\em semiselector} $A \in \mathcal F$ i.e. $|A \cap I_n| \leq 1$ for all $n < \omega$,
\item
$\mathcal F$ is {\em Ramsey} if it is both a $P$-set and a $Q$-set.
\end{enumerate}
If $\mathcal F$ is a filter and a $P$-set (respectively a $Q$-set, Ramsey set) we call $\mathcal F$ a $P$-filter (respectively a $Q$-filter, Ramsey filter).
\end{definition}

We can now give the definition of a selective maximal independent family.

\begin{definition}
An independent family $\mathcal{I}$ is \emph{selective} if it is densely maximal and $\mathrm{fil}(\mathcal{I})$ is Ramsey.
\end{definition}

Selective independent families were first introduced by Shelah in \cite{Sh92}, where it is shown in the course of his proof of the consistency of $\mfi < \mfu$ that under $\CH$ there is a selective independent family. Since then the following results have been shown concerning the forcing indestructibility of selective independent families. 

\begin{fact}\label{fact.all}
Let $\mathcal I$ be a selective independent family. Then $\mathcal I$ is remains selective (and hence maximal) after forcing with a countable support product of Sacks forcing  (Shelah, see \cite[Theorem 4.6]{DefMIF} or~\cite[Corollary 37]{FMIdeals}). Moreover, $\mathcal{I}$ remains selective independent after forcing with the countable support iteration of any of the following:
\begin{enumerate}
\item Sacks forcing (Shelah, see \cite[Theorem 4.6]{DefMIF} or \cite[Corollary 37]{FMIdeals});
%\item
%countable support products of Sacks forcing (Shelah, see \cite[Theorem %4.6]{DefMIF} or~\cite[Corollary 37]{FM19});
\item forcing notions of the form $\mathbb Q_\mathcal I$ from Shelah's~\cite{Sh92}; 
\item Miller partition forcing (see~\cite{CFGS21});
\item $h$-Perfect Tree Forcing Notions for different functions $h:\omega \to \omega$ with $1 < h(n) < \omega$ for all $N < \omega$ (see~\cite{CS});
\item Coding with perfect trees (see \cite{BFS21});
\item Miller lite forcing, (see \cite[Theorem 4.1]{FS23});
\item any mix of the above (a consequence of \cite[Theorem 3.8]{FS23}).
\end{enumerate}
\end{fact}

Obviously all the forcing constructions described above produce a model where the selective independent family is of size $\aleph_1$. As stated in the introduction, the purpose of this paper in part is to show that consistently there are selective families of other sizes.

We finish these preliminaries by remarking that there are in fact maximal, non-densely maximal independent families. This is due to Juris Stepr\=ans and is included with his kind permission.

\begin{proposition}[Stepr\=ans]
    For any cardinal $\kappa$ for which there is a maximal independent family there is a maximal, non densely maximal independent family of size $\kappa$. In particular there is always one of size $\mfi$. \label{Juris}
\end{proposition}

\begin{proof}
Fix a cardinal $\kappa$ for which there is a maximal independent family. Let $Z \subseteq \omega$ be an infinite, co-infinite set. By translation there is a maximal independent family $\mathcal I = \{A_\alpha \; | \; \alpha < \kappa \}$ on $Z$ (so all the $A_\alpha$'s are subsets of $Z$). Let $\mathcal J = \{B_\alpha \; | \; \alpha < \kappa\}$ be an independent, but not maximal independent family on $\omega \setminus Z$. Note that there is always an independent family of size continuum (\cite[Proposition 8.9]{BlassHB}) so in particular there is one of size $\kappa$. Finally let $\mathcal K = \{A_\alpha \cup B_\alpha \; | \; \alpha < \kappa\} \cup \{Z\}$. It is routine to check that this is independent. Moreover if $X \notin \mathcal K$ then either $X \cap Z$ is almost disjoint, in which case $\mathcal K \cup \{X\}$ is not independent or $X \cap Z = A_\alpha$ for some $\alpha$ in which case $\mathcal K \cup \{X\}$ is not independent or else there is a Boolean combination $h$ on $\mathcal I$ witnessing that $X \cap Z$ cannot be added to $\mathcal I$ and hence $\mathcal K \cup \{X\}$ is not independent. Therefore $\mathcal K$ is maximal.

Finally let $Y \subseteq \omega \setminus Z$ be a set so that $Y \notin \mathcal J$ but $\mathcal J \cup \{Y\}$ is independent. Observe that as a result no Boolean combination $h \in \mathsf{FF}(\mathcal K)$ extending $\langle Z, 1\rangle$ will be such that $\mathcal K^{h} \setminus Y$ or $\mathcal K^{h} \cap Y$ is finite. Thus $\mathcal K$ is not densely maximal.
\end{proof}

This counterexample should be contrasted with the combined content of Lemmas 6.6 and 6.7 of \cite{GS94} where it is shown that every maximal independent family is densely maximal below some Boolean combination i.e. for each m.i.f. $\mathcal I$ there is an $h \in \mathsf{FF}(\mathcal I)$ so that the set $\{A \cap \mathcal I^h \; | \; A \in \mathcal I \setminus {\rm dom}(h)\}$ is densely maximal as a family on $\mathcal I^h$. Note that as a corollary of this it follows that in $\ZFC$ there are densely maximal independent families and indeed they exist in every cardinality for which there is a maximal independent family.  

\section{The Selectivity of the Generic Maximal Independent Family}

Next we prove Theorem \ref{mainthm}. Let us fix some notation for the rest of this section. Let $\kappa$ be an ordinal of uncountable cofinality. Let $\mathcal I_0$ be some fixed independent family and $\mathcal F_0$ be a diagonalization filter. Now inductively let $\langle \P_\alpha, \dot{\Q}_\alpha\; | \; \alpha < \kappa\rangle$ be a finite support iteration and $\dot{\mathcal I}_\alpha$ and $\dot{\mathcal F}_\alpha$ be $\P_\alpha$-names defined as follows. 
\begin{enumerate}
    \item $\P_0$ is the trivial forcing, $\dot{\Q}_0$ is the trivial name for $\mathbb M(\mathcal F_0)$.
    \item $\P_\alpha$ forces that $\dot{\mathcal I}_\alpha$ is an independent family with a diagonalization filter $\dot{\mathcal F}_\alpha$.
    \item $\P_{\alpha + 1} = \P_\alpha * \mathbb M (\dot{\mathcal F}_\alpha)$.
    \item $\P_{\alpha + 1}$ forces that $\dot{\mathcal I}_{\alpha + 1} = \dot{\mathcal I}_\alpha \cup \{\dot{m}_\alpha\}$ where $\dot{m}_\alpha$ is the name for the $\mathbb M (\dot{\mathcal F}_\alpha)$-generic real. 
    \item If $\beta$ is a limit ordinal then $\P_\beta$ forces that $\dot{\mathcal I}_\beta = \bigcup_{\gamma < \beta} \dot{\mathcal I_\gamma}$.
\end{enumerate}
Let $\dot{\mathcal I}_\kappa$ be the $\P_\kappa$-name for the union of the $\dot{\mathcal I}_\alpha$'s. Let $G_\kappa \subseteq \P_\kappa$ be generic over $V$ and, in $V[G_\kappa]$ let $\mathcal I_\alpha$, $\mathcal F_\alpha$, $m_\alpha$ etc refer to the evaluation of all of the corresponding names with the dots. Finally let for each $\alpha < \kappa$ the generic $G_\alpha = G_\kappa \cap \P_\alpha$ as usual. For the rest of this section we fix all such objects. We refer to $\mathcal I_\kappa$ (in $V[G_\kappa]$) as {\em the generic m.i.f.}. Let us now restate Theorem \ref{mainthm} more precisely.

\begin{theorem}
    $\P_\kappa$ forces that the generic m.i.f. is densely maximal and its density filter is a $P$-point. Moreover, if $\kappa$ is a cardinal and $\kappa^{{<}\kappa} = \kappa$ in the ground model then the $\dot{\mathcal F}_\alpha$'s can be chosen so that the generic mif is selective. \label{mainthmbetter}
\end{theorem}

It is unclear whether the above can be improved so as to eliminate the need to choose the filters so as to ensure that ${\rm fil}(\mathcal I_\kappa)$ is a $Q$-filter. We will discuss this more later. To prove Theorem \ref{mainthmbetter} we need to show three things: that $\mathcal I_\kappa$ is densely maximal, that its density filter is a $P$-point and that its density filter is a $Q$-point given good enough bookkeeping. We will prove each of these separately, beginning with dense maximality.

\subsection{Dense Maximality}
Again, we fix the notation described in the first paragraph of this section. 
\begin{lemma}
    $\P_\kappa$ forces that the generic m.i.f. is densely maximal. \label{dense}
\end{lemma}

\begin{proof}
%Let us be clear about the set up. Fix a cardinal $\kappa$ of uncountable cofinality and define a finite support ccc iteration as follows. First let $\mathcal I_0$ be some fixed, ground model independent family, $\P_0$ be $\mathbb M(\mathcal I_0)$ (for any diagonalization filter). Next assume $\P_\alpha$ is given, as is $\dot{\mathcal I}_\alpha$ and let $\dot{\Q}_\alpha$ be a $\P_\alpha$-name for $\mathbb M(\dot{\mathcal I}_\alpha)$ and let $\dot{\mathcal I}_{\alpha + 1}$ name the independent family where $\dot{\mathcal I}_\alpha$ is extended by $\dot{m}_\alpha$ the generic real added by $\dot{\mathbb Q}_\alpha$. Let $\dot{\mathcal I}_\kappa$ be the union of the previous $\dot{\mathcal I}_\alpha$'s. %Moreover arrange this choice of diagonalization filters so that for every nice name for a real $\dot{X}$ in some diagonalization filter for $\dot{\mathcal I}_\kappa$ then there are unboundedly many $\alpha$'s with $\dot{X}$ forced to be in the diagonalization filter used at stage $\alpha$. The details of how this is managed are well known and we omit the nitty gritty details. 

By Lemma \ref{mainlemma} it suffices to show that if $X \in V[G_\kappa]$ has infinite intersection with every Boolean combination then it is in the density filter of $\mathcal I_\kappa$ as this implies that the density filter is the unique diagonalization filter and hence $\mathcal I_\kappa$ is densely maximal. So suppose that $X \in V[G_\kappa]$ has infinite intersection with every Boolean combination and let $h \in \mathsf{FF}(\mathcal I)$. We need to find an $h' \supseteq h$ so that $\mathcal I^{h'}_\kappa \setminus X$ is finite. Let $\alpha < \kappa$ be such that $X, h \in V[G_\alpha]$. 

\noindent \underline{Case 1:} There is a $\beta \geq \alpha$ so that $X$ is forced to be in $\dot{\mathcal F}_\beta$, the diagonalization filter used at stage $\beta$. Now the generic real $m_{\beta} \in \mathcal I_\kappa$ is a pseudointersection of this filter and in particular $m_{\beta} \subseteq^* X$. But then we get that if $h' = h \cup \langle m_{\beta}, 0\rangle$ then $\mathcal I_{\beta}^{h'} \subseteq^* m_{\beta} \subseteq^* X$ so $h'$ is as needed. Note that this was valid since $h \in V[G_\alpha]$ and so in particular $\beta \notin {\rm dom}(h)$. 

\noindent \underline{Case 2:} $X$ is not forced to be in any diagonalization filter at any stage $\beta \geq \alpha$. Work $V[G_\alpha]$ and let $\mathcal F_\alpha$ be the choice of diagonalization filter for $\mathcal I_\alpha$. Note that the assumption implies in particular that $X$ is not in $\mathcal F_\alpha$. Since $\mathcal F_\alpha$ is maximal with the property that every element has infinite intersection with every Boolean combination of $\mathcal I_\alpha$, there must be a $Y \in \mathcal F_\alpha$ and a Boolean combination $g \in \mathsf{FF}(\mathcal I_\alpha)$ so that $X \cap Y \cap \mathcal I^g_\alpha$ is finite. But now note that $\dot{m}_\alpha$ is forced to be an almost subset of $Y$ thus we get that $V[G_{\alpha+1}] \models$``$X \cap \dot{m}_{\alpha}^{G_{\alpha + 1}} \cap \mathcal I_\alpha^g$ is finite". But $\dot{m}_{\alpha}^{G_{\alpha + 1}} \cap \mathcal I_\alpha^g$ is a Boolean combination of $\mathcal I_\kappa$, contradicting the defining property of $X$. 
\end{proof}

Having established dense maximality we go on to consider the property of the density filter being a $P$-filter.

\subsection{${\rm fil}(\mathcal I_\kappa)$ is a $P$-Filter}
We continue with our notation outlined above.

\begin{lemma}
The density filter of the generic m.i.f. added by $\P_\kappa$ is a P-filter. \label{pfilter}
\end{lemma}

\begin{proof}
Let $\{\dot{A}_n \; | \; n < \omega\}$ name an $\omega$-sequence of elements of ${\rm fil}(\mathcal I_\kappa)$. By the fact that $\kappa$ has uncountable cofinality there is a $\gamma < \kappa$ so that $\dot{A}_n \in V[G_\gamma]$ and, moreover, by Lemma \ref{lemma1} we can find such a $\gamma$ so that $\dot{A}_n \in {\rm fil}(\mathcal I_\gamma)$ for all $n < \omega$. Work in such a $V[G_\gamma]$ and let $A_n$ be the evaluation of $\dot{A}_n$ in this model. 

Since $A_n \in {\rm fil}(\mathcal I_\gamma)$ for each $n < \omega$ we must have that for each $n < \omega$ the set $A_n$ is in every diagonalization filter we choose from stage $\gamma$ on, again by Lemma \ref{lemma1}. Consequently for all $\xi > \gamma$ we have that $m_\xi \subseteq^* A_n$ for all $n < \omega$. In particular (working now in $V[G_{\gamma + \omega}]$) we have $m_{\gamma + n} \subseteq^* A_m$ for all $n, m < \omega$. For each $k < \omega$ let $l_k(n)$ be such that $m_{\gamma + n} \setminus l_k(n) \subseteq A_k$. Let $f:\omega \to \omega$ dominate all of the $l_k$'s. Finally set $$B = \bigcup_{n < \omega} (m_{\gamma + n} \setminus f(n))$$

We claim that $B \subseteq^* A_n$ for each $n < \omega$ and $B$ is forced to be in the density filter for $\mathcal I_\kappa$, which completes the proof. For the first part fix $k < \omega$ and let $m$ be such that for all $n > m$ we have $f(n) > l_k(n)$. Now we have $$B = \bigcup_{n \leq m} (m_{\gamma + n} \setminus f(n)) \cup \bigcup_{n > m} (m_{\gamma + n} \setminus f(n))$$
Observe that $\bigcup_{n \leq m} m_{\gamma + n} \subseteq^* A_k$ since it is a finite union of almost subsets of $A_k$ and $\bigcup_{n > m} m_{\gamma + n} \setminus f(n) \subseteq A_k$ (true inclusion - not mod finite) since $f(n) > l_k(n)$ and by definition of $l_k$ we have that $m_{\gamma + n} \setminus l_k(n) \subseteq A_k$. Putting these two observations together proves the first part of our claim, namely that $B \subseteq^* A_k$.

For the second part let $h \in \mathsf{FF}(\mathcal I_\kappa)$. Since $h$ is finite there is an $n < \omega$ so that $m_{\gamma + n} \notin {\rm dom}(h)$. Fix such an $n < \omega$ and let $h' = h \cup \langle m_{\gamma + n}, 0\rangle$. Now $\mathcal I_\kappa^{h'} = \mathcal I_\kappa^h \cap m_{\gamma + n} \subseteq m_{\gamma + n} \subseteq^* B$ so $B$ is in the density filter as needed.
\end{proof}

\begin{remark}
By one of the results of \cite{FMIdeals} we know that forcing with $\mathbb{M}(\mathcal I)$ for $\mathcal I$ selective adds a dominating real. It follows that, once we have proved Lemma \ref{qfilter} below and hence Theorem \ref{mainthmbetter} in Mathias iterations as we have been describing we get that $\mfb = \mfd = {\rm cof}(\kappa)$. The existence of such dominating reals allows us to sup up the argument for \ref{pfilter} (since the only place we used countability was to get a dominating function) so we will have shown that in fact, assuming we choose the diagonalization filters correctly, ${\rm fil}(\mathcal I_\kappa)$ is a $P_{{\rm cf}(\kappa)}$-filter i.e. any ${<}{\rm cof}(\kappa)$-many elements have a pseudointersection in the filter.
\end{remark}

The proof of Lemma \ref{pfilter} actually shows that a basis for the density filter of $\mathcal I_\kappa$ in $V[G_\kappa]$ is given by simply $\bigcup_{n < \omega} (m_{\xi_n} \setminus f(n))$ for functions $f:\omega \to \omega$ and elements $m_{\xi_n} \in \mathcal I_\kappa$ (with infinitely many distinct). Extracting from this we get the following.

\begin{lemma}
    Let $A \in [\omega]^\omega \cap V[\mathcal I_\kappa]$. The following are equivalent.
    \begin{enumerate}
        \item There is a $\gamma < \kappa$ so that for all $\alpha \in (\gamma, \kappa)$ we have $m_\gamma \subseteq^* A$.
        \item There are strictly increasing ordinals $\xi_n < \kappa$ for $n < \omega$ so that for all $n < \omega$ we have $m_{\xi_n} \subseteq^* A$.
        \item $A \in {\rm fil}(\mathcal I_\kappa)$.
    \end{enumerate} \label{basis}
\end{lemma}

\begin{proof}
Fix $A$ as above and work in $V[\mathcal I_\kappa]$. Since (1) is an obvious strengthening of (2) we have (1) implies (2) and (2) implies (3) is exactly as in the proof of Lemma \ref{pfilter}. Thus it suffices to show that (3) implies (1). Observe that by the ccc there is a $\gamma < \kappa$ so that $A \in [\omega]^\omega \cap V[\mathcal I_\gamma]$. Moreover since $A \in {\rm fil}(\mathcal I_\kappa)$ we can assume without loss of generality that $A \in {\rm fil}(\mathcal I_\gamma)$ (the $\gamma$ where $A$ first appears might be before the one in which $A$ ends up in the density filter but we just take the latter in this case). Now we get that $A \in \mathcal F_\alpha$ for each $\alpha > \gamma$ hence $m_\alpha \subseteq^* A$ for every such $\alpha$ since $m_\alpha$ is a pseudointersection of $\mathcal F_\alpha$. 
\end{proof}

\subsection{${\rm fil}(\mathcal I_\kappa)$ is a $Q$-Filter}

Finally we will show that ${\rm fil}(\mathcal I_\kappa)$ is a $Q$-filter when the diagonalization filters $\mathcal F_\alpha$ are chosen carefully enough. As noted in the hypotheses of Theorem \ref{mainthmbetter}, we will eventually need that $\kappa$ is a cardinal and $\kappa^{{<}\kappa} = \kappa$ but we will state this explicitly when we need it. For now we proceed with the notation given above. We will use the following characterization of $Q$-filters.
\begin{fact}[See Lemma 3.7 of \cite{CFGS21}]
Let $\mathcal F$ be a filter on $\omega$. The following are equivalent.
\begin{enumerate}
\item
$\mathcal F$ is a $Q$-filter.
\item
For every strictly increasing $f:\omega \to \omega$ there is a $A \in \mathcal F$ so that, letting $A = \{k(n)\}_{n < \omega}$ be an increasing enumeration of $A$, $f(k(n)) < k(n+1)$.
\end{enumerate}
\end{fact}
Moving forward, in the pursuit of brevity, if $f$ and $A$ have the property described in (2) above we will say that $A$ $Q$-{\em dominates} $f$. We need one more fact.
\begin{lemma}
In $V[G_\alpha]$ there is a diagonalization filter $\mathcal F' \supseteq \mathcal F_\alpha$ for $\mathcal I_{\alpha + 1}$. Consequently we can always choose the diagonalization filters for the iteration to be a $\subseteq$-strictly increasing sequence. \label{biggerfilters}
\end{lemma}

\begin{proof}
Work in $V[G_\alpha]$ and fix $X \in \mathcal F_\alpha$. It suffices to show that, under the hypotheses we have that $\forces_{\mathbb M(\mathcal F_\alpha)}$ ``$\check{X}$ has infinite intersection with every $h \in \mathsf{FF}(\dot{\mathcal I}_{\alpha + 1})$" as in this case every $X \in \mathcal F_\alpha$ has infinite intersection with every Boolean combination of $\mathcal I_{\alpha + 1}$ (note that we are assuming the maximal condition forces this) so we can extend $\mathcal F_\alpha$ to some diagonalization filter for $\mathcal I_{\alpha+1}$. Fix an arbitrary condition $(s, A)$ and an arbitrary $h \in \mathsf{FF}(\mathcal I_{\alpha + 1})$. If $m_{\alpha} \notin {\rm dom}(h)$ then by hypothesis we have $X \cap \mathcal I_{\alpha+1}^h$ is infinite so assume $m_{\alpha} \in {\rm dom}(h)$ and let $h' = h \setminus \{\langle m_{\alpha}, h(m_\alpha)\rangle \}$. We need to show that for each $n < \omega$ and each $i < 2$ we have that there is a $k > n$ with $k \in \mathcal I_{\alpha}^{h'} \cap X \cap m_\alpha^i$. Fix $n < \omega$, and $i < 2$. Without loss we can assume that ${\rm dom}(s) \supsetneq n$ and $A \subseteq X$ as the set of conditions whose stem has domain containing $n$ is dense and, since $X \in \mathcal F_\alpha$ we can always replace $A$ by $A \cap X$ if we so choose. Note that again we have that $A \cap \mathcal I^{h'}_\alpha$ has infinite intersection, and therefore there is in particular a $k \in A \cap \mathcal I_{\alpha}^{h'}$ and $k > n$ since $(s, A)$ is a condition and therefore ${\rm min}(A) > n$. Let $s' \supseteq s$ so that $s'(k) = i$. Then $(s', A \setminus {\rm dom}(s'))$ is a condition extending $(s, A)$ which forces that $k \in \mathcal I_{\alpha}^{h'} \cap A \cap m_\alpha^i \subseteq \mathcal I_{\alpha}^{h'} \cap X \cap m_\alpha^i$, so we are done.
\end{proof}

\begin{lemma}
    If $\gamma < \kappa$ and $f \in \baire \cap V[G_\gamma]$ is strictly increasing, then there is a choice of diagonalization filters $\dot{\mathcal F}_{\gamma + i}$ for $i \in \omega + \omega$ so that (using this choice of filters for the Mathias forcing notions) we have that in $V[G_{\gamma + \omega + \omega}]$ there is an $A \in {\rm fil}(\mathcal I_{\gamma + \omega + \omega})$ which $Q$-dominates $f$. \label{Qdominate}
\end{lemma}

\begin{proof}
Fix $\gamma$ and $f$ as in the hypothesis of the lemma. Observe by the finiteness of the support, a density argument ensures that for every $k < \omega$ there is an $n < \omega$ so that ${\rm min}(m_{\gamma + n}) > k$. Using this, inductively define $A = \{k(n)\}_{n <\omega}$ so that $k(n+1) = {\rm min}(m_{\gamma + n_{n+1}})$ where $n_{n+1}$ is the least number $l$ so that ${\rm min} (m_{\gamma + l}) > f(k(n))$. Clearly $A$, which is in $V[G_{\gamma + \omega}]$, $Q$-dominates $f$ so it remains to show that the next $\omega$-many diagonalization filters can be chosen so that $A$ is forced to be in ${\rm fil}(\mathcal I_{\gamma + \omega + \omega})$.

Work in $V[G_{\gamma + \omega}]$. 

\begin{claim}
    $A$ has infinite intersection with every Boolean combination in $\mathcal I_{\gamma + \omega}$. \label{claim1}
\end{claim}

\begin{proof}[Proof of Claim \ref{claim1}]
    This is a density argument. Suppose that $k < \omega$, and $h \in \mathsf{FF}(\mathcal I_{\gamma + \omega})$. Let $n$ be large enough that ${\rm dom}(h) \subseteq \mathcal I_{\gamma + n}$ and let $m$ be such that the first $m$-entries of $A$ are the minimums of elements from among $m_{\gamma + l}$ for $l < n$ and denote these elements $\{k(j)\}_{j < m}$. Without loss of generality we can assume that $m > k$. Work in $V[G_{\gamma + n}]$. We now let $a > f(k(m-1))$ be in $\mathcal I_{\gamma + n}^h$. Since this set is infinite such an $a < \omega$ exists. Finally let $s \in 2^{<\omega}$ be the sequence of length $a + 1$ so that for all $b < a$ we have $s(b) = 0$ and $s(a) = 1$. Clearly, regardless of the choice of $\mathcal F_{\gamma + n}$ we have that $(s, \omega) \in \mathbb M(\mathcal F_{\gamma + n})$ and forces that the $m^{\rm th}$-element of $A$ is in $\mathcal I_{\gamma + n}^h\setminus k$. Since $k$, $h$ and $n$ were arbitrary, the proof is complete.
\end{proof}

Given Claim \ref{claim1}, observe that we can put $A$ into $\mathcal F_{\xi}$ for $\xi \in [\gamma + \omega, \gamma + \omega + \omega)$. The first step i.e. putting $A$ in $\mathcal F_{\gamma + \omega}$ follows from the claim since $A$ has infinite intersection with every Boolean combination. The following steps follow from Lemma \ref{biggerfilters}. Therefore $A \in {\rm fil}(\mathcal I_{\gamma + \omega+\omega})$ and hence $A \in {\rm fil}(\mathcal I_\kappa)$ by Lemma \ref{basis}, thus completing the proof.
\end{proof}

Now we can prove the following lemma which implies Theorem \ref{mainthmbetter}.
\begin{lemma}
If $\kappa^{{<}\kappa} = \kappa$ is a cardinal then there is a choice of diagonalization filters so that $\P_\kappa$ forces that ${\rm fil}(\dot{\mathcal I}_\kappa)$ is a $Q$-filter. \label{qfilter}
\end{lemma}

\begin{proof}
We want to show that we can choose the diagonalization filters so that every strictly increasing $f \in \baire \cap V[G_\kappa]$ is $Q$-dominated by some $A \in {\rm fil}(\mathcal I_\kappa)$. By Lemma \ref{Qdominate} we can ensure for any {\em fixed} $f \in \baire \cap V[G_\kappa]$ this can be done but then the cardinal arithmetic hypothesis, alongside the ccc of the forcing ensures that there is enough space to handle every $f$ with some bookkeeping as there are only $\kappa$-many nice names for reals. 
\end{proof}

As stated before the combination of Lemmas \ref{dense}, \ref{pfilter} and \ref{qfilter} prove Theorem \ref{mainthmbetter} (and hence Theorem \ref{mainthm}).

\section{Arbitrarily large selective independent families}

Since Mathias forcing notions relativized to a filter are all $\sigma$-centered, by \cite[Theorem 7.12]{BlassHB}, we get the following, which strengthens a theorem of Shelah from \cite{Sh92}, who proved the same under the stronger hypothesis of $\CH$ in place of $\mfp = \cc$.
\begin{theorem}
    Assume $\mfp = \cc$. Every independent family $\mathcal I_0$ of size ${<}\cc$ can be extended to a selective independent family. \label{p=ccor}
\end{theorem}

\begin{proof}
Fix $\mathcal I_0$, an independent family of some size $\lambda < \cc$. Enumerate the elements of $\mathcal I_0$ as $\{A_\xi \; | \; \xi < \lambda\}$. For $\mathcal I_0$, and further independent families we will build in this proof, we associate a finite partial function $h:\cc \to 2$ to a Boolean combination by mapping e.g. $A_\alpha$ to $A_\alpha^{h(\alpha)}$. We will not comment on this again and assume implicitly that some enumeration of our independent families has been chosen to make sense of this. If there is a $\zeta \in {\rm dom}(h)$ which is greater than $\lambda$ then we consider the Boolean combination undefined. Enumerate all pairs consisting of an element of Ramsey space and a finite partial function $h:\cc \to 2$ as $\{(X_\alpha, h_\alpha) \; | \; \alpha < \cc\}$, enumerate countable subsets of Ramsey space $\{(A^\alpha_n) \; | \; n < \omega, \, \alpha < \cc\}$ so that every sequence appears unboundedly often and fix a scale $\{f_\alpha \; | \; \alpha < \cc\} \subseteq \baire$ (so $\alpha < \beta$ implies $f_\alpha \leq^* f_\beta$ and this family is dominating). Note that the assumption on $\mfp$ guarantees such a scale exists. We will inductively define a continuous, $\subseteq$-increasing sequence of independent families $\{\mathcal I_\alpha\}_{\alpha < \cc}$ so that the union $\bigcup_{\alpha < \cc} \mathcal I_\alpha$ is selective. Indeed it suffices to show that given $\mathcal I_\alpha$ independent we can find $\mathcal I_{\alpha + 1} \supseteq \mathcal I_\alpha$ so that the following hold:

   \begin{enumerate}
       \item If $X_\alpha$ has infinite intersection with every Boolean combination of $\mathcal I_\alpha$ and $h_\alpha$ is defined on $\mathcal I_\alpha$ then there is an $h' \supseteq h_\alpha$ so that $\mathcal I_\alpha^{h'} \setminus X_\alpha$ is finite.
       \item If $(A_n^\alpha) \subseteq {\rm fil}(\mathcal I_\alpha)$ then there is a $B \in {\rm fil}(\mathcal  I_{\alpha + 1})$ so that for all $n < \omega$ we have $B \subseteq^* A_n^\alpha$.
       \item There is a $C \in {\rm fil}(\mathcal I_{\alpha+1})$ which $Q$-dominates $f_\alpha$.
   \end{enumerate}

   The rest of the proof is standard bookkeeping argument. So fix $\alpha < \cc$. At each stage we will add at most countably many reals so we can assume that $\mathcal I_\alpha$ has size ${<} \cc$. We will deal with the three requirements in three steps. In the first step, if $X_\alpha$ does not have infinite intersection with every Boolean combination of $\mathcal I_\alpha$ or $h_\alpha$ is not defined on $\mathcal I_\alpha$ then we do nothing and let $\mathcal I_\alpha^0 = \mathcal I_\alpha$. Otherwise we use the forcing axiom characterization of $\mfp = \cc$ applied to $\mathbb M (\mathcal F_\alpha)$ where $\mathcal F_\alpha$ is a diagonalization filter containing $X_\alpha$. By ${<}\cc$-many dense sets we can find a $Y$ which is independent over $\mathcal I_\alpha$ and $Y \subseteq^* X$ since there are ${<}\cc$-many Boolean combinations. Let $\mathcal I^0_\alpha = \mathcal I_\alpha \cup \{Y\}$ and note that by the same argument as in the proof of Lemma \ref{dense} this satisfies criterion (1) above.

   Next, if $(A_n^\alpha) \nsubseteq {\rm fil}(\mathcal I_\alpha)$, let $\mathcal I^1_\alpha = \mathcal I^0_\alpha$. Otherwise, by successively choosing diagonalization filters (which will all have all the $A_n^\alpha$'s) we can find countably many sets $(Y_n)_{n < \omega}$ so that $\mathcal I^1_\alpha := \mathcal I^0_\alpha \cup \{Y_n\; | \; n < \omega\}$ is independent and each $Y_n$ is an almost subset of every $A^\alpha_n$. As in the proof of Lemma \ref{pfilter}, for each $k$ let $l_k(n)$ be such that $Y_n \setminus l_k(n) \subseteq A^\alpha_n$ and let $f\in \baire$ dominate all the $l_k$'s. The same proof as in Lemma \ref{pfilter} ensures then that $B = \bigcup_{n < \omega} (Y_n \setminus f(n))$ is in ${\rm fil}(\mathcal I^1_\alpha)$ and $B \subseteq^* A^\alpha_n$ for all $n < \omega$ as needed for criterion (2).

   Finally for criterion (3), we can again use the forcing axiom characterization of $\mfp = \cc$ applied in this case to mimic the proof of Lemma \ref{Qdominate} to find a $Z$ which $Q$-dominates $f_\alpha$ and has infinite intersection with every Boolean combination in $\mathcal I_\alpha^1$. Finally similar to the proof of criterion (2) in the previous paragraph we can find sets $\{Z_n\; | \; n < \omega\}$ so that $\mathcal I_{\alpha + 1} : = \mathcal I^2_\alpha = \mathcal I^1_\alpha \cup \{Z_n\; | \; n < \omega\}$ is independent and $Z_n \subseteq^* Z$ for all $n < \omega$. By the same proof again as in Lemma \ref{Qdominate}, this ensures that $Z \in {\rm fil}(\mathcal I_{\alpha + 1})$, thus completing the construction and hence the proof. 
\end{proof}

We also get the following. 

\begin{theorem}
    Let $\kappa < \lambda$ be cardinals both of uncountable cofinality. It is consistent that $\cc = \lambda$ and there is a selective independent family of size $\kappa$. Moreover if $\kappa$ is regular we can have that $\mfi = \kappa$ i.e. the selective independent family is of minimal size. \label{i<c}
\end{theorem}

\begin{proof}
By forcing if necessary assume that $\cc = \lambda$. Let $\mu = {\rm cf}(\kappa)$ and let $\{i_\alpha \; | \; \alpha < \mu\}$ be a $\mu$-length cofinal sequence. We will force that $\mfb = \mfd = \mu$ and there is a selective independent family of size $\kappa$. Since $\mfd \leq \mfi$ in $\ZFC$, see \cite[Theorem 8.13]{BlassHB}, in the case $\kappa$ is regular this will complete the proof of the ``moreover" part as well. Now, define a finite support iteration $\langle \P_\alpha, \dot{\Q}_\alpha\; | \;\alpha < \kappa\rangle$ so that if $\alpha \notin \{i_\xi \; | \; \xi < \mu\}$ then $\P_\alpha$ forces that $\dot{\Q}_\alpha$ is the Mathias forcing for some inductively defined independent family as in the construction described in Theorem \ref{mainthmbetter}. If $\alpha$ is some $i_\xi$ then let $\mathbb P_\alpha$ force that $\dot{\Q}_\alpha$ is Hechler forcing, followed by the $\omega + \omega$ stage iteration described in Lemma \ref{Qdominate} to make the Hechler Generic $Q$-dominated by some element of the filter of the family we are adding.

Let $\mathcal I_\kappa$ be the generic independent family added by this iteration. By the arguments in the previous section, it is clear that this family will be densely maximal and have a density filter which is a $P$-filter. Moreover, the Hechler reals will form a scale of length $\mu$ (and every set of reals of size ${<}\mu$ will be dominated by some Hechler real) hence $\mfb = \mfd = \mu$. Also, each Hechler real will be $Q$-dominated by some element of ${\rm fil}(\mathcal I_\kappa)$. Since the family of Hechler reals is dominating this is enough to ensure that ${\rm fil}(\mathcal I_\kappa)$ is a $Q$-filter thus completing the proof. 
\end{proof}

\section{Conclusion and Open Questions}

We conclude this paper with a list of questions for further research. %There still remain many unanswered questions about selective independent families. 
The most important of these, as mentioned in the introduction is the following.

\begin{question}
    Is there always a selective independent family? If there is one, is there always one of size $\mfi$? \label{existence}
\end{question}

Towards answering this question we note that very little is even known about the existence of selective independent families in models where the ground model selective independent families are not preserved. Indeed, until the current paper we did not know if $\mfi = \cc>\aleph_1$ was consistent with the existence of a selective independent family. In particular we would like to know:
%That the following simple test question is open proves our ignorance in even beginning to solve Question \ref{existence}.
\begin{question}
    Are there selective independent families in the Cohen model?
\end{question}

Turning our attention to the results of this paper point out to the following loose end from the proof of Theorem \ref{mainthmbetter}: Did we need to choose the diagonalization filters to ensure the $Q$-filter property? More precisely of interest is the following:

\begin{question}
    Can an iteration of Mathias forcing as described above produce an independent family whose diagonalization filter is not a $Q$-filter?
\end{question}

%Finally we remark that in the proof of Theorem \ref{i<c} the use of Hechler reals seems essential. Therefore the following seems like it would be an important line of inquiery to understand the $Q$-filter property better.
%\begin{question}
%    Suppose $\mathcal I$ is selective, is ${\rm cf}(|\mathcal I|) \geq \mathfrak b$?
%\end{question}
%
%Of course this relates to the well known question of whether it is consistent that $\mfi$ %has countable cofinality.

\end{document}